\documentclass[12pt,reqno]{article}

\usepackage[usenames]{color}
\usepackage{amssymb}
\usepackage{amsmath}
\usepackage{amsthm}
\usepackage{amsfonts}

\usepackage[colorlinks=true,
linkcolor=webgreen,
filecolor=webbrown,
citecolor=webgreen]{hyperref}

\definecolor{webgreen}{rgb}{0,.5,0}
\definecolor{webbrown}{rgb}{.6,0,0}

\usepackage{fullpage}
\usepackage{float}

\usepackage{breakurl}

\newcommand{\seqnum}[1]{\href{https://oeis.org/#1}{\rm \underline{#1}}}

\newtheorem{theorem}{Theorem}

\date{}

\begin{document}


\begin{center}
\vskip -0.4in
Journal of Integer Sequences, Vol. 24 (2021),\\
Article 21.10.5
\vskip 0.4in
\end{center}

\begin{center}
\vskip 1cm{\LARGE\bf
On the Solution of the Equation $n = ak + bp_k$ \\
\vskip .1in
by Means of an Iterative Method
}
\vskip 1cm
Juan Luis Varona\footnote{The work of the author is supported by the spanish
Minis\-te\-rio de Cien\-cia, Inno\-va\-ci\'on y Uni\-ver\-si\-da\-des 
under Grant PGC2018-096504-B-C32.}\\
Departamento de Matem\'aticas y Computaci\'on\\
Universidad de La Rioja\\
Logro\~no\\
Spain\\
\href{mailto:jvarona@unirioja.es}{\tt jvarona@unirioja.es}
\end{center}

\vskip .2in

\begin{abstract}
For fixed positive integers $n$, we study the solution of the equation
$n = k + p_k$,
where $p_k$ denotes the $k$th prime number, by means
of the iterative method
\[
  k_{j+1} = \pi(n-k_j), \qquad k_0 = \pi(n),
\]
which converges to the solution of the equation, if it exists.
We also analyze the equation $n = ak + bp_k$
for fixed integer values $a \ne 0$ and $b>0$,
and its solution by means of a corresponding iterative method.
The case $a>0$ is somewhat similar to the case $a=b=1$, 
while for $a<0$ the convergence and usefulness of the method 
are less satisfactory.
The paper also includes a study of the dynamics of the iterative methods.
\end{abstract}

\section{Introduction and main results}

Let us imagine that we want to know if a Fibonacci number, a pentagonal number, one of a pair of amicable numbers or a Pythagorean triple, the terms in a sequence in the On-line Encyclopedia of Integer Sequences (OEIS, \cite{OEIS}), or any other positive integer can be written as
\begin{equation}
\label{eq:n=k+pk}
  n = k + p_k,
\end{equation}
where $(p_k)_{k \ge 1}$ is the sequence of the prime numbers. (Let us note that neither Mathematica, Maple nor SageMath are able to solve the equation~\eqref{eq:n=k+pk} even in simple cases with small numbers, for instance for $n=10+p_{10}=39$.)

Of course, one can look at the sequence $(k+p_k)_{k \ge 1}$ (which is the sequence \seqnum{A014688} in the OEIS) and check if $n$ occurs in the sequence; but this can be a hard task. Instead, we give here an iterative process which, given $n$, tells whether the equation \eqref{eq:n=k+pk} has a solution or not, and, in the affirmative case, provides the solution~$k$.

It is worth noting that there exist a large number of iterative methods for solving equations of the form $f(x) = 0$ where $f$ is a function defined on $\mathbb{R}$, $\mathbb{C}$, $\mathbb{R}^d$, $\mathbb{C}^d$ or a Banach space. However, there are very few iterative methods for solving equations in the theory of numbers. 
Perhaps the best example of an iterative method for solving an equation in this area is \cite{KnXe}, where the authors show how classical rootfinding methods from numerical analysis can be used to calculate inverses of units modulo prime powers.

Along the paper, we will use some well known properties of prime numbers, which can be found in many texts; see, for instance, \cite{Apo, BatDia, DKoLu, FiRo, New}. 
As usual, $\pi(x)$ denotes the quantity of prime numbers $\le x$, which is a nondecreasing function. Moreover, we will often use the basic properties $\pi(p_k) = k$ and $p_{\pi(n)} \le n$ (with $p_{\pi(n)} = n$ if and only if $n$ is prime). Let us also note that $k + p_k$ is strictly increasing with~$k$. Thus, for a fixed $n$ the solution of~\eqref{eq:n=k+pk}, if it exists, is unique.

A possible method to try to find $k$ such that $n = k + p_k$ is to take $f(k) = k + p_k - n$ and solve the equation $f(k)=0$ by means of a bisection method starting at the initial points $k_0 = 1$ and $k_1=n$. The function $f(k)$ is increasing and satisfies $f(k_0)<0$ and $f(k_1)>0$ so, if the solution exists, the bisection method always finds it. Moreover, the number of iterations needed to reach the solution is $\lfloor\log_2(n-1)\rfloor$ (of course, this number can be a bit different due to the rounding of the bisection to an integer in every step). Instead, we propose another method which is faster, as we can see in the examples, and has a nice dynamics with its own interest.

The iterative method that we are going is simple and, as far as we know, has not been proposed before in this context; it is a kind of regula falsi method to solve an equation $f(k) = 0$ with $f(k) = \pi(n-k) - k$, that is not equivalent to~\eqref{eq:n=k+pk}, but closely related.
Since $p_k \approx k \log k$, it follows that $p_k$ gives the main contribution to the sum $k+p_k$. 
The idea is then to approximate the equation \eqref{eq:n=k+pk} by the equation $n = p_k$, and so to take a first guess $k = \pi(n)$; then we adjust the initial guess by successive corrections. 
Let us note that any $k$ which solves $n=k+p_k$ satisfies $n-k=p_k$, a prime number, so $\pi(n-k) = k$.
Thus, the procedure for solving~\eqref{eq:n=k+pk} is based on the iterative scheme $k_{j+1} = \pi(n-k_j)$, and the goal is to look for a fixed point.
The following result shows the dynamics of the iterative method:

\begin{theorem}
\label{theo:iter}
Let $n \ge 4 $ be a given integer, and let us define the iterative process
\begin{equation}
\label{eq:iter}
  k_{j+1} = \pi(n-k_j), \qquad k_0 = \pi(n).
\end{equation}
After a finite number of steps, one of these cases occurs:
\begin{enumerate}
\item[(i)] We get a fixed point $k^{*}$.
\item[(ii)] We get a cycle $\{k',k''\}$ with $k'' = k'+1$, $n-k'$ a prime number, $(k_{2j})_{j\ge0}$ is a decreasing sequence which converges to $k''$, and $(k_{2j+1})_{j\ge0}$ is an increasing sequence which converges to~$k'$.
\end{enumerate}
\end{theorem}

We postpone the proof of this theorem to Section~\ref{sec:proofs}.
For the moment, let us see the consequences of this result on the solutions of the equation~\eqref{eq:n=k+pk}.

\begin{theorem}
\label{theo:sol}
Let $n \ge 4$ be a given integer, the equation $n = k + p_k$,
and the iterative process
\begin{equation*}
  k_{j+1} = \pi(n-k_j), \qquad k_0 = \pi(n).
\end{equation*}
Under these circumstances:
\begin{enumerate}
\item[(i)] If we get a fixed point $k^{*}$ (case (i) in Theorem~\ref{theo:iter}), then $k^{*}$ is the solution of the equation
\[
  n' = k + p_k
\]
where $n' = \max\{ j + p_j \le n : j \ge 1\}$. 
In particular, the equation $n = k + p_k$ has a solution if and only if $n'=n$, 
and then $k^{*}$ is the solution.
(In practice it is enough to check whether $k^{*} + p_{k^{*}} = n$ or not.)
\item[(ii)] If we get a cycle $\{k',k''\}$ (with $k''=k'+1$, case (ii) in Theorem~\ref{theo:iter}), the equation $n = k + p_k$ has no solution.
\end{enumerate}
\end{theorem}

As in the case of Theorem~\ref{theo:iter}, we postpone the proof of Theorem~\ref{theo:sol} to Section~\ref{sec:proofs}.

Both cases (i) and (ii) can really occur, as the following illustrative examples with small numbers show. The case (i) occurs, for instance, with $n=51$ or $n=76$. 
For $n=51$, the successive $k_j$ in \eqref{eq:iter} are
\begin{gather*}
  k_0 = \pi(51) = 15,
  \quad
  k_1 = \pi(51-15) = 11,
  \\
  k_2 = \pi(51-11) = 12,
  \quad
  k_3 = \pi(51-12) = 12,
\end{gather*}
so we have reached the fixed point $k^{*} = 12$; but $k^{*} + p_{k^{*}} = 12 + 37 = 49 \ne 51$,
so the equation $51 = k + p_{k}$ has no solution.
For $n=76$, the successive $k_j$ in \eqref{eq:iter} are
\begin{gather*}
  k_0 = \pi(76) = 21,
  \quad
  k_1 = \pi(76-21) = 16,
  \\
  k_2 = \pi(76-16) = 17,
  \quad
  k_3 = \pi(76-17) = 17,
\end{gather*}
so $k^{*} = 17$ is a fixed point; this time $k^{*} + p_{k^{*}} = 17+59 = 76$,
so we have found the solution of $76 = k + p_{k}$.
To illustrate the case (ii), let us take, for instance, $n=41$.
The successive $k_j$ in \eqref{eq:iter} are
\begin{gather*}
  k_0 = \pi(41) = 13,
  \quad
  k_1 = \pi(41-13) = 9,
  \quad
  k_2 = \pi(41-9) = 11,
  \\
  k_3 = \pi(41-11) = 10,
  \quad
  k_4 = \pi(41-10) = 11,
  \\
  k' = k_5 = k_7 = k_9 = \cdots = 10,
  \quad
  k'' = k_6 = k_8 = k_{10} = \cdots = 11.
\end{gather*}
Then, $41 = k + p_{k}$ has no solution.

Some examples with bigger numbers are given in Table~\ref{tab:Fib}.
The iterative algorithm is applied to some Fibonacci numbers (\seqnum{A000045} in the OEIS), showing whether the method converges to a fixed point $k^{*}$ or to a cycle $\{k',k''\}$, as well as the number of iterations to reach it. Observe that the number of iterations for $n \le F_{70} \approx 1.9\cdot 10^{14}$ is always less than ten; for each equation, we also see how our method is faster than the bisection method, which requires a considerably higher number of iterations.

\begin{table}
\centering
\begin{tabular}{c@{}c@{}c@{}cc}
\hline
$F_m$ & Iterations & $k^{*}$ or $\{k',k''\}$ 
& Is $k^{*}$ a solution? & $\lfloor\log_2{F_{m}}\rfloor$
\\
\hline
$F_{12} = 144$ & 3 & $30$ & no & 7
\\
$F_{13} = 233$ & 3 & $\{42, 43\}$ & --- & 7
\\
$F_{14} = 377$ & 4 & $\{64, 65\}$ & --- & 8
\\
$F_{15} = 610$ & 3 & $97$ & no & 9
\\
$F_{27} = 196\,418$ & 5 & $\{16\,347, 16\,348\}$ & --- & 17
\\
$F_{42} = 267\,914\,296$ & 7 & $13\,887\,473$ & yes & 27
\\
$F_{50} = 12\,586\,269\,025$ & 8 & $\{543\,402\,114, 543\,402\,115\}$ & --- & 33
\\
$F_{53} = 53\,316\,291\,173$ & 7 & $2\,166\,313\,972$ & yes & 35
\\
$F_{66} = 27\,777\,890\,035\,288$ & 9 & $899\,358\,426\,281$ & no & 44
\\
$F_{67} = 44\,945\,570\,212\,853$ & 9 & $1\,432\,816\,693\,546$ & no & 45
\\
$F_{68} = 72\,723\,460\,248\,141$ & 9 & $2\,283\,240\,409\,254$ & no & 46
\\
$F_{69} = 117\,669\,030\,460\,994$ & 9 & $3\,639\,256\,692\,076$ & no & 46
\\
$F_{70} = 190\,392\,490\,709\,135$ & 9 & $5\,801\,907\,791\,391$ & no & 47
\\
\hline
\end{tabular}%
\caption{The iterative method applied to solve $F_m = k + p_k$ where $F_m$ are the Fibonacci numbers with $12 \le m \le 70$.
We omit the $F_{m}$ with $15 < m \le 65$ whose corresponding iterative methods converge to a fixed point $k^{*}$ which is not a solution of $F_m = k + p_k$. In the last column, $\lfloor\log_2{F_{m}}\rfloor$ is a rather precise estimation of the number of iterations required to solve the equation with the bisection method}.
\label{tab:Fib}
\end{table}

Instead of~\eqref{eq:n=k+pk} we can consider the more general equation
\begin{equation*}
  n = ak + bp_k,
\end{equation*}
with $b \ge 1$ and $a \in \mathbb{Z} \setminus \{0\}$
(the trivial case $a=0$ is excluded).

In this setting we can consider the iterative process which starts with $k_0 = \pi(n/b)$ and continues with
\[
  k_{j+1} = \pi((n-ak_j)/b), 
  \qquad j \ge 0.
\]
Now, the behavior of the iterative method and its relation with the solution of $n = ak + bp_k$ is a bit more complicated than in the case $a=b=1$. We analyze it in Section~\ref{sec:gen}.

\section{Proof of Theorems~\ref{theo:iter} and~\ref{theo:sol}}
\label{sec:proofs}

\begin{proof}[Proof of Theorem~\ref{theo:iter}]
It is clear that $k_1 = \pi(n-k_0) \le \pi(n) = k_0$. From the inequality $k_1 \le k_0$, it follows that $k_2 = \pi(n-k_1) \ge \pi(n-k_0) = k_1$. Now, from $k_2 \ge k_1$ we get $k_3 = \pi(n-k_2) \le \pi(n-k_1) = k_2$, and so on. That is,
\begin{equation}
\label{eq:alterna}
  k_{2j+1} \le k_{2j}
  \quad\text{and}\quad
  k_{2j+2} \ge k_{2j+1},
  \qquad \text{for }j = 0,1,2,\dots.
\end{equation}
Moreover, $k_2 = \pi(n-k_1) \le \pi(n) = k_0$, and from the inequality $k_2 \le k_0$ it follows that $k_3 = \pi(n-k_2) \ge \pi(n-k_0) = k_1$. Now, from $k_3 \ge k_1$ we get $k_4 = \pi(n-k_3) \le \pi(n-k_1) = k_2$, and so on. Thus, $(k_{2j})_{j\ge0}$ is a decreasing sequence, while $(k_{2j+1})_{j\ge0}$ is an increasing sequence. They are both bounded sequences of positive integers, so eventually constant; that is, there exist $k'$ and $k''$ (with $k' \le k''$) and a certain $J$ such that
\[
  k_{2j+1} = k' 
  \quad\text{and}\quad
  k_{2j} = k'' 
  \qquad \text{for }j \ge J.
\]
If $k' = k''$, then $k^{*} = k'=k''$ is the fixed point in~\eqref{eq:iter}, the possibility (i) in the theorem. 

Otherwise, let us suppose that $k' < k''$, so
\[
  k' = \pi(n-k'') < \pi(n-k') = k''.
\]
Then $\pi(n-k')-\pi(n-k'')=k''-k'$, so there are $k''-k'$ prime numbers $q_j$ such that
\begin{equation}
\label{eq:qj}
  n-k'' < q_1 < q_2 < \cdots < q_{k''-k'-1} < q_{k''-k'} \le n-k'.
\end{equation}
Except for $2$ and $3$, prime numbers are not consecutive numbers, so \eqref{eq:qj} is only possible in two cases:
\begin{itemize}
\item Case $k''-k'=2$, with $n-k''=1$, $n-k'=3$, $q_1 = 2$ and $q_2=3$. 
In this case, we would have $k' = \pi(n-k'') = \pi(1) = 0$, which cannot happen because $k'$ is the limit of the increasing sequence $(k_{2j+1})_{j\ge0}$.
\item Case $k''-k'=1$, with a unique prime $q_1 = n-k'$ in~\eqref{eq:qj}. Then $k''=k'+1$ and
\[
  \pi(n-k') = k'+1, \quad \pi(n-k'-1) = k'.
\]
This is the possibility (ii) in the theorem.
\qedhere
\end{itemize}
\end{proof}

\begin{proof}[Proof of Theorem~\ref{theo:sol}]
Let us first analyze the case (i). 
We have $k^{*} = \pi(n-k^{*})$ and $p_{k^{*}} \le p_{\pi(n-k^{*})} \le n-k^{*}$, so $k^{*} + p_{k^{*}} \le n$. Thus, $k^{*} \in A$ where $A$ is the set
\[
  A = \{ j \ge 1 : j + p_j \le n \}.
\]
To conclude the proof of case (i), it is enough to check that $\max A = k^{*}$ (in particular, this implies $k^{*} + p_{k^{*}} = n'$). 
Let us suppose that $\max A \ne k^{*}$. In this case, $k^{*}+1 \in A$, so $p_{k^{*}+1} + k^{*}+1 \le n$, and therefore $p_{k^{*}+1} \le n - k^{*} - 1 \le n - k^{*}$. This implies that $\pi(n-k^{*}) \ge k^{*} + 1$, which is false because $\pi(n-k^{*}) = k^{*}$.

In case (ii), we have
\[
  k' = \pi(n-k'') < \pi(n-k') = k'',
\]
with $k'' = k'+1$.
Using that $n-k'$ is a prime number, we have
$p_{k'+1} = p_{k''} = p_{\pi(n-k')} = n-k'$, so $k' + p_{k'+1} = n$ and then
\[
k' + p_{k'} < n < k'+1 + p_{k'+1}.
\]
This clearly implies that $n = k + p_k$ has no solution.

Finally, let us check that there cannot exist any solution of \eqref{eq:n=k+pk} which is 
not detected by the iterative method~\eqref{eq:iter}. 
Let us suppose on the contrary that $k$ is a solution not detected by the method;
then, $n = k + p_k$ and $\pi(n-k) = k$, so $k$ is a fixed point of
$k_{j+1} = \pi(n-k_j)$, that is, a fixed point of \eqref{eq:iter} except for the starting step $k_0 = \pi(n)$.

But we have proved that the iterative method~\eqref{eq:iter}, starting 
in~$k_0 = \pi(n)$, converges to $k^{*}$ or to the 2-cycle $\{k',k''\}$
with $k''=k'+1$. If \eqref{eq:iter} converges to $k^{*}$, this $k^{*}$ is the unique solution of
\eqref{eq:n=k+pk}, so $k=k^{*}$.
If \eqref{eq:iter} converges to $\{k',k''\}$, then $k<k'$ or $k>k''$. 
However, both cases are impossible:
\begin{equation}
\label{eq:kk'k''}
\begin{gathered}
k < k' \;\Rightarrow\; k = \pi(n-k) \ge \pi(n-k') = k'' \;\Rightarrow\; k \ge k'', \text{ a contradiction};
\\
k > k'' \;\Rightarrow\; k = \pi(n-k) \le \pi(n-k'') = k' \;\Rightarrow\; k \le k', \text{ a contradiction}.
\end{gathered}
\end{equation}
Then, there cannot exist such fixed point $k$.
\end{proof}

\section{Generalization to the equation $n = a k + b p_k$}
\label{sec:gen}

A more general equation than \eqref{eq:n=k+pk} is
\begin{equation}
\label{eq:n=ak+bpk}
  n = a k + b p_k
\end{equation}
with $b \in \mathbb{Z}^{+}$ and $a \in \mathbb{Z} \setminus\{0\}$.
Given $n$, we want to know if there exists some $k$ satisfying~\eqref{eq:n=ak+bpk}
and how to find it by the iterative process 
\begin{equation}
\label{eq:iterab}
  k_{j+1} = \pi((n-ak_j)/b), 
  \qquad j \ge 0,
\end{equation}
starting with $k_0 = \pi(n/b)$.

Of course, the reason to choose the iterative process \eqref{eq:iterab} is that,
if $k$ is a solution of \eqref{eq:n=ak+bpk}, then $(n-ak)/b = p_k$, a prime number,
so $\pi((n-ak)/b) = \pi(p_k) = k$ and $k$ is a fixed point of~\eqref{eq:iterab}.

Depending on whether $a>0$ or $a<0$, the beginning of the iterative process \eqref{eq:iterab} varies. 
If $a>0$ we have $k_1 \le k_0$; however, if $a<0$ we have $k_1 \ge k_0$. 
Moreover, if $a<0$, the sequence $ak+bp_k$ is not always increasing with $k$, so we cannot ensure that the solution of $n = ak+bp_k$, if it exists, is unique. These differences motivate to consider both cases $a>0$ and $a<0$ separately.

In what follows, we present this study in a more informal way, without stating the properties as theorems. Actually, the behavior of the iterative method and its relation with the solutions of \eqref{eq:n=ak+bpk} is similar to what happens in Theorems~\ref{theo:iter} and~\ref{theo:sol} in the case $a>0$, but rather different in the case $a<0$.

\subsection{Case $a>0$}
\label{sec:gena>0}

Let us start noticing that we want the $k_j$ that arise in the iterative method~\eqref{eq:iterab} to be positive integers (otherwise, $p_{k_j}$ does not exist). 
Due to the equivalence $\pi(x) \sim x/\log(x)$ for $x\to\infty$ (prime number theorem),
we can guarantee that $k_j > 0$ for all $j$ if $n$ is large enough
(the exact required size depends on $a$ and~$b$). Indeed, for $n\to\infty$ we have
\[
  k_0 = \pi(n/b) \sim \frac{n/b}{\log(n/b)} \sim \frac{n}{b\log(n)}
\]
and
\[
  k_1 = \pi((n-ak_0)/b) \sim \frac{(n-ak_0)/b}{\log((n-ak_0)/b)} 
  \sim \frac{n}{b\log(n)},
\]
so $k_0, k_1 \ge 1$ if $n$ is big enough. 
Then, given the equation $n = a k + b p_k$, we can assume that $n$ satisfies 
\begin{equation}
\label{eq:admis}
  k_1 = \pi\big((n-ak_0)/b\big) = \pi\big((n-a\pi(n/b))/b\big)  \ge 1
\end{equation}
(this holds except for a finite set of $n$'s). Note that, for $a=b=1$, \eqref{eq:admis} becomes $\pi(n-\pi(n)) \ge 1$, which holds for every $n \ge 4$.
After imposing this technical restriction for small values of $n$, let us analyze the iterative method.

The assumption $a>0$ gives $k_1 = \pi((n-ak_0)/b) \le \pi(n/b) = k_0$. From $k_1 \le k_0$, it follows that $k_2 = \pi((n-ak_1)/b) \ge \pi((n-ak_0)/b) = k_1$, and from $k_2 \ge k_1$, we get $k_3 = \pi((n-ak_2)/b) \le \pi((n-k_1)/b) = k_2$, and so on. 
Moreover, $k_2 = \pi((n-ak_1)/b) \le \pi(n/b) = k_0$
and $k_3 = \pi((n-ak_2)/b) \ge \pi((n-ak_0)) = k_1$, and so on.
In particular, if $k_1 \ge 1$ then $k_j \ge 1$ for every $j$, 
so assuming \eqref{eq:admis} is enough to ensure that all the $k_j$ will be positive integers.

Following as in the proof of Theorem~\ref{theo:iter} we get that $(k_{2j})_{j\ge0}$ is a decreasing sequence, $(k_{2j+1})_{j\ge0}$ is an increasing sequence, and there exist $k'$ and~$k''$ (with $k' \le k''$) and a certain $J$ such that
\[
  k_{2j+1} = k' 
  \quad\text{and}\quad
  k_{2j} = k'' 
  \qquad \text{for }j \ge J.
\]
If $k' = k''$, then $k^{*} = k'=k''$ is the fixed point in~\eqref{eq:n=ak+bpk}. 
Otherwise, if $k' < k''$ we have
\[
  k' = \pi((n-ak'')/b) < \pi((n-ak')/b) = k''
\]
so $\pi((n-ak')/b)-\pi((n-ak'')/b)=k''-k'$, and there are $k''-k'$ prime numbers $q_j$ such that
\begin{equation}
\label{eq:qjab}
  \frac{n-ak''}{b} < q_1 < q_2 < \cdots < q_{k''-k'-1} < q_{k''-k'} \le \frac{n-ak'}{b}.
\end{equation}
{}From now on, the case $k' < k''$ is somewhat different.

Let us suppose that two of these primes are $q_1 = 2$ and $q_2 = 3$. Then $(n-ak'')/b = 1$ and $k' = \pi((n-ak'')/b) = \pi(1) = 0$, which cannot happen. Therefore, we can assume that $q_{j+1}-q_j \ge 2$ for all these primes~$q_j$. Clearly, any interval $(x,y]$ containing $m$ of these primes must satisfy $y-x > 2(m-1)$. In the case \eqref{eq:qjab}, this means that
\[
  \frac{n-ak'}{b} - \frac{n-ak''}{b} > 2(k''-k')-2,
\]
that is, $a(k''-k')/b > 2(k''-k')-2$, or
\begin{equation}
\label{eq:2b-a2b}
  (2b-a)(k''-k') < 2b.
\end{equation}
Then, two situations can occur:
\begin{itemize}
\item If $a \ge 2b$, \eqref{eq:2b-a2b} is trivial and does not imply any restriction on $k'$ and~$k''$.
\item If $a < 2b$, \eqref{eq:2b-a2b} can be written as
\begin{equation}
\label{eq:2b/2b-a}
  k''-k' < \frac{2b}{2b-a}.
\end{equation}
In the particular case $a \le b$ we have $1 < 2b/(2b-a) \le 2$, so \eqref{eq:2b/2b-a} is equivalent to 
$k'' = k'+1$ (recall that we are analyzing the case $k'<k''$). 
\end{itemize}

In the case $a \le b$, the remaining arguments in the proofs of Theorems~\ref{theo:iter} and~\ref{theo:sol} are valid. In particular \eqref{eq:kk'k''}, which guarantees that the iterative method \eqref{eq:iterab}, starting at $k_0 = \pi(n/b)$, detects as fixed points all the solutions of $n = ak + bp_k$.

If $a \ge 2b$ or $b < a < 2b$, we cannot ensure that $k''=k'+1$. There might be an integer $k$ with $k' < k < k''$ which is a fixed point
\[
  \pi((n-ak)/b) = k
\]
and a solution of $n = ak + bp_k$ not detected by the iterative method \eqref{eq:iterab} starting in $k_0=\pi(n/b)$. That is, the method converges to the 2-cycle $\{k',k''\}$ instead of $k$. In practice, if we are looking for a solution of $n = ak + bp_k$, it is enough to check if every $k$ satisfying $k' < k < k''$ is a solution.

Let us give some examples to illustrate that these situations can occur.

For the case $a \ge 2b$, let us see what happens with the equation
\[
  n = 7k+2p_k
\]
for several values of~$n$, with convergence to a fixed point $k^{*}$ or to cycles $\{k',k''\}$ with different values of $k''-k'$. 
For $n=10\,040$, we have $k_0 = 672$ and the iterative process converges to $k^{*} = 474$, which is a solution of the equation; 
instead, for $n = 10\,041$, again $k_0 = 672$ and the iterative process converges to $k^{*} = 474$, which is not a solution of the equation.
For $n=10\,073$, we have $k_0 = 674$ and the iterative process converges to the cycle $\{474,476\}$.
For $n=10\,300$, we have $k_0 = 686$ and the iterative process converges to the cycle $\{482,485\}$.
For $n=10\,325$, we have $k_0 = 687$ and the iterative process converges to the cycle $\{483,487\}$.
For $n=10\,532$, we have $k_0 = 698$ and the iterative process converges to the cycle $\{491,497\}$.
In all these cases, less that $10$ iterations are enough to reach $k^{*}$ or $\{k',k''\}$.

For the case $b < a < 2b$, in the equation $12\,660 = 3k+2p_k$ we have $k_0 = 824$ and the iterative method converges to the cycle $\{699,701\}$, where $k''-k'=2$. 

Let us also show some instances of the equation $n = ak+bp_k$ with a solution $k$ such that the iterative method \eqref{eq:iterab} starting with $k_0=\pi(n/b)$ converges to a cycle $\{k',k''\}$. 
For $n = 2\cdot 33 + p_{33} = 203$, where $k=33$ is clearly a solution of $n = 2k + p_{k}$, the iterative method converges to the cycle $\{32, 34\}$, with $k''-k'=2$.
For $n = 6\cdot 100 + p_{100} = 1141$, where $k=100$ is the solution of $n = 6k + p_{k}$, the iterative method converges to the cycle $\{80, 121\}$, with $k''-k'=41$.

\subsection{Case $a<0$}
\label{sec:gena<0}

Let us start noticing that, in general, $ak + bp_k$ is not increasing with $k$ when $a<0$. 
Therefore, we cannot ensure the unicity of the solution of $n = ak + bp_k$ for fixed $a$, $b$ and~$n$.
For instance, $n=105$ has two solutions $k$ in the equation $n = -2k+p_k$:
\[
  105 = -2 \cdot 43 + p_{43} = -2 \cdot 44 + p_{44}.
\]
With $n = -3k+p_k$ we can also find two solutions that are not consecutive integers: 
\[
  100 = -3 \cdot 59 + p_{59} = -3 \cdot 61 + p_{61}.
\]
There can exist more than two solutions, as is the case of the equation $n = -4k+p_k$:
\[
  99 = -4 \cdot 83 + p_{83} = -4 \cdot 85 + p_{85} = -4 \cdot 86 + p_{86}.
\]
However, $p_k$ always grows faster than $k$, so the sequence $ak + bp_k$ is increasing with $k$ if $0 < -a < b$. 
In this case, again the solution of $n = ak + bp_k$, if it exists, is unique.

Anyway, taking into account that $p_k \sim k \log k$ when $k \to \infty$, we have
$ak + bp_k \sim ak + bk\log k = (a+b\log k)k$. 
Thus, $ak + bp_k$ is increasing with $k$ for $k$ big enough. 
In particular, this ensures that, for fixed $a$, $b$ and~$n$, the number of solutions of $n = ak + bp_k$ is always finite.
Precise estimates of the form
\[
  C_1 \, \frac{x}{\log x} \le \pi(x) \le C_2 \, \frac{x}{\log x}
\]
with $C_1, C_2 > 0$, which yield an upper bound (depending on $a$, $b$, $n$, $C_1$ and $C_2$) of the number of solutions, can be found in many texts of number theory (see, for instance, \cite[Theorem~8.8.1]{BaSh}).

Now, let us analyze the behavior of the iterative method \eqref{eq:iterab} starting in $k_0 = \pi(n/b)$.
As shown below, nothing similar to \eqref{eq:alterna} appears in the case $a<0$;
the property \eqref{eq:alterna} is very important in the analysis of the case $a>0$,
and the lack of a suitable alternative is a great handicap.
Although the analysis of the case $a<0$ is presented here for completeness, it must be concluded that
the iterative method described in this paper is not so useful to solve~\eqref{eq:n=ak+bpk}
as it is when $a>0$.

Since $a<0$, we have $k_1 = \pi((n-a k_0)/b) \ge \pi(n/b) = k_0$; 
and, if we assume that $k_j \ge k_{j-1}$, also
\[
  k_{j+1} = \pi((n-a k_j)/b) \ge \pi((n-a k_{j-1})/b) = k_j.
\]
Then $k_j$ is an increasing sequence and it cannot tend to a cycle.

Let $s$ be any number such that 
\begin{equation}
\label{eq:cotaa<0}
  \pi((n-as)/b) \le s;
\end{equation}
the existence of such an $s$ follows from the estimate $\pi(x) \sim x/\log x$ when $x \to \infty$.
{}Then (recall that $a<0$)
\[
  k_0 = \pi(n/b) \le \pi((n-as)/b) \le s;
\]
and assuming that $k_j \le s$ gives 
\[
  k_{j+1} = \pi((n-ak_j)/b) \le \pi((n-as)/b) \le s.
\]
This proves that the increasing sequence of integers $k_j$ produced by the iterative method \eqref{eq:iterab} starting with $k_0 = \pi(n/b)$ is bounded.
So there exists some~$k^{*}$ (and an index $J$) such that
\[
  k_j = k^{*} \quad\text{for every } j \ge J,
\]
and thus $k^{*}$ is a fixed point of~\eqref{eq:iterab}.

We cannot ensure that the fixed point $k^{*}$ is a solution of \eqref{eq:n=ak+bpk}.
A solution of \eqref{eq:n=ak+bpk}, if it exists, satisfies $\pi((n-as)/b) = s$;
in particular, it also satisfies \eqref{eq:cotaa<0}. 
Then $k^{*} \le s$ for any possible solution $s$ of \eqref{eq:n=ak+bpk} such that $s \ge \pi(n/b) = k_0$.
In practice, we can continue checking whether $k^{*}$, $k^{*}+1$, $k^{*}+2$,\dots\ are solutions or not 
of \eqref{eq:n=ak+bpk} until, when substituting in $ak+bp_k$, we get a value greater than~$n$; 
actually, an additional precaution is necessary: 
we must continue the checking until reaching values of $k$ for which the sequence $ak+bp_k$ is already increasing.

For an example of this behavior, let us take $n = -7 \cdot 2000 + p_{2000} = 3389$ (that is, $a=-7$ and $b=1$), 
so $s = 2000$ is a solution of the equation $3389 = -7k+p_k$. However, the iterative process $k_{j+1} = \pi(3389+7k_j)$ 
starting in $k_0 = \pi(3389) = 477$ converges to $k^{*} = 1989$, which is not a solution of the equation. Observe that $s-k^{*} = 11$.

Finally, let us observe an extra surprise: in the case $a<0$, the fixed point $k^{*}$ of the iterative method 
is never a solution of $n = ak+b_k$. If the equation has no solutions, there is nothing to prove. 
If the equation has a solution $s$, let us prove that the fixed point $k^{*}$ cannot be $k^{*} = s$, but $k^{*} < s$. 
Indeed, it is enough to check that, if $k_j < s$, also $k_{j+1} < s$. 
We have $(n-ak_j)/b < (n-as)/b = p_s$, a prime, so $\pi((n-ak_j)/b) < s$.
But $\pi((n-ak_j)/b) = k_{j+1}$, so $k_{j+1} < s$, just as we were looking for.

\section{Acknowledgments}

I am grateful to Samuel Stigliano, an Uruguayan student (of medicine!), 
for the motivation to study this problem. 
In March 2021, he told me by email that both numbers of the amicable pair
$(220, 284)$ can be written as $k + p_k$, 
namely $220 = 41 + p_{41}$ and $284 = 51 + p_{51}$, and he asked if
a similar property happens with other pairs of amicable numbers
(\seqnum{A259180} in the OEIS).
Trying to answer his question, I wanted to decide if a number
$n$ can be written as $n = k + p_{k}$ with a fast algorithm,
and then I arrived at the iterative method presented in this paper.
In particular, this allowed me to find some pairs of amicable numbers 
that satisfy the requested property; the smaller is 
\[
1\,392\,368 = 99\,525 + p_{99\,525},
\quad
1\,464\,592 = 104\,283 + p_{104\,283}.
\]
This pair of amicable numbers was discovered by Euler in 1747.



\bigskip
\hrule
\bigskip

\noindent 2020 \emph{Mathematics Subject Classification: }
Primary 11B83;
Secondary 11A41, 11Y16, 37C25, 65H05.

\noindent \emph{Keywords: } 
integer sequence; prime number; iterative method to solve equations; dynamics.

\bigskip
\hrule
\bigskip

\noindent (Concerned with sequences
\seqnum{A014688}, \seqnum{A000045}, and \seqnum{A259180}.)

\bigskip
\hrule
\bigskip

\vspace*{+.1in}
\noindent
Received March 21 2021; 
revised versions received March 23 2021; November 16 2021; November 21 2021.
Published in \emph{Journal of Integer Sequences}, November 27 2021.

\bigskip
\hrule
\bigskip

\noindent
Return to
\htmladdnormallink{Journal of Integer Sequences home page}{http://www.cs.uwaterloo.ca/journals/JIS/}.
\vskip .1in

\end{document}